\documentclass[reqno, 11pt]{amsart}

\usepackage{srcltx}
\usepackage[mathscr]{eucal}
\usepackage{amscd}
\usepackage{color}
\usepackage{amsfonts}
\usepackage{amsmath, amsthm, amssymb, ,empheq}
\usepackage{latexsym}
\usepackage{enumerate}
\usepackage{graphics}
\usepackage{epsfig}
\usepackage{xcolor}
\usepackage{graphicx}
\usepackage{epstopdf}
\usepackage{dsfont}
\usepackage{tikz}
\usepackage{nicefrac}

\usepackage{multicol}
\usepackage{enumerate}
\usepackage{hyperref}
\usepackage{siunitx}
\usepackage{todonotes}
\usepackage{cleveref}
 \newtheorem{thm}{Theorem}[section]
 
 \newtheorem{lem}[thm]{Lemma}

 \newtheorem{defn}[thm]{Definition}
 \newtheorem{rem}[thm]{Remark}
 
 \numberwithin{equation}{section}

\def\R{{\Bbb R}}

\def\R{{\mathbb R}}
\def\N{{\Bbb N}}

\def\bge{\begin{eqnarray}}
\def\bgee{\begin{eqnarray*}}
\def\ege{\end{eqnarray}}
\def\egee{\end{eqnarray*}}

\newcommand{\ee}{{\mathrm{e}}}

\newcommand{\e}{\varepsilon}

\newcommand{\dive}{\operatorname{div}}
\newcommand{\di}[1]{\,\mathrm{d}#1}

\newcommand{\reach}{\operatorname{reach}}
\newcommand{\dist}{\operatorname{dist}}

\def\be{\begin{equation}}
\def\ee{\end{equation}}
\def\bse{\begin{subequations}}
\def\ese{\end{subequations}}
\def\bge{\begin{eqnarray}}
\def\bgee{\begin{eqnarray*}}
\def\ege{\end{eqnarray}}
\def\egee{\end{eqnarray*}}

\newcommand{\bfr}{\bf\color{red}}

\newcommand{\clr}{\color{red}}

\usepackage{lineno} 

\date{\today}

\title[Colloids deposition in porous media]
{A mathematical model  for colloids deposition in porous media combined with a moving boundary at the  microscale: {S}olvability and numerical simulation}

\author{Christos Nikolopoulos}
\address{Department of Mathematics, University of Aegean,
GR-83200 Karlovassi, Samos, Greece}

\author{Michael Eden}
\address{Department of Mathematics, University of Regensburg, Germany}

\author{Adrian Muntean}
\address{Department of Mathematics and Computer Science,
Karlstad University, Sweden}

\date{\today}

\begin{document}
\maketitle


\begin{abstract}
{\clr} We study a reaction-diffusion model posed on two distinct spatial scales that accounts for diffusion, aggregation, fragmentation, and deposition of populations of colloidal particles within a porous material.
In this model, the macroscopic transport of the particles is described by an effective equation whose transport coefficients are determined by cell problems posed on the underlying pore scale.
 The internal pore geometry can change over time due to deposition or detachment of colloidal particles.
We represent the evolving microstructure as solid cores whose phase boundaries can grow or shrink over time.
As deposition progresses, neighbouring growing cores may come into contact, leading to local clogging of the pore space.
We investigate how such evolving microstructures influence the effective transport and storage properties of porous layers.
We establish basic analytical results concerning the weak solvability of the resulting multiscale evolution problem, which takes the form of a strongly non-linear parabolic system, in the non-clogging regime. 
For the numerical approximation of weak solutions we propose a two-scale finite element discretization. Numerical experiments illustrate how local clogging affects the effective dispersion tensor and quantify the resulting trade-off between transport efficiency and storage capacity. 
\end{abstract}

\maketitle
\vspace{0.5in}

\section{Introduction}

Functional porous materials (active filters, biological membranes, etc.) derive their macroscopic properties from coupled physical processes operating across multiple length and time scales. Pore-scale mechanisms, including diffusion, flow, ionic transport, chemical reactions, deposition, or adsorption, interact with microstructures such as phase domains and interfaces (phase boundaries), giving rise to emergent macroscopic responses.
Such interactions are hard to control as usually the microscopic information is not evenly distributed, and hence, geometric defects can be present (either because they were manufactured in the production stage or because they appear due to clogging, or some other type of localized damage).
A predictive understanding of such materials requires modeling frameworks that explicitly connect fine-scale physics to effective material behavior.
This is precisely where one major knowledge gap is.
What is currently unclear (and sometimes controversial) is precisely how the coupling between the scales must be in general achieved. A better understanding of such microscopic-macroscopic transfer of information would allow scientists to better engineer the functionality of their materials.

Specifically, if the material's length scales of interest are strongly separated, then, for specific situations,  
variations of homogenization-based  techniques are able to unveil the structure of the expected upscaled model, 
if one is provided with suitable descriptions of the microscopic geometries as well as with the set of pore-scale mechanisms that are relevant for the physical situation at hand. Even in this somewhat simpler case, it is often difficult (or out of reach) to incorporate all the wanted information in a couple of effective (transport, sorption, etc.) coefficients (cf. e.g. \cite{Gruy,Ivan}).
In that case, one hopes that computationally-feasible multiscale models are a good remedy (see \cite{Showalter_Oberwolfach} for a nice introduction to the topic of models with distributed microstructures and \cite{otero_multiscale_2018} for an overview of such models from an engineering point of view).
This direction of thinking naturally leads to a couple of interesting fundamental questions within the framework of partial differential equations posed on multiple scales, like their solvability, large-time behavior of their solutions, or the computability of the said solutions.

The problem setting we have here in mind is described in section \ref{PuwS}. The structure of these equations and the microscopic-macroscopic coupling are derived formally in the work \cite{MC20}  by combining locally-periodic homogenization arguments with matched asymptotics; see also \cite{WiedemannPeter2023} for an alternative rigorous upscaling.
The existence of weak solutions to the two-scale model and their numerical approximation for a selection of fixed microstructures are discussed in our more recent works \cite{EdenNikolopoulosMuntean22, NEM23}, with the results restricted to the 2D case while the porous microstructures are represented by disks. 
Such prior investigations and also the mathematical analysis shown in this work are focused on the non-clogging regime -- the solid cores might get very close but are not yet fully in contact with the cell boundaries.
In our numerical simulations, however, we specifically look at the case of clogging: as microstructures grow due to the deposition of populations of colloidal particles, local clogging occurs, that is neighbouring cores may touch each other (which is represented by the solid cores coming into contact with the cell boundaries).
Within this work, we investigate how distributions of evolving microstructures influence the overall transport and storage properties of porous layers.

\begin{figure}[hbtp]
\centering
\includegraphics[scale=.8]{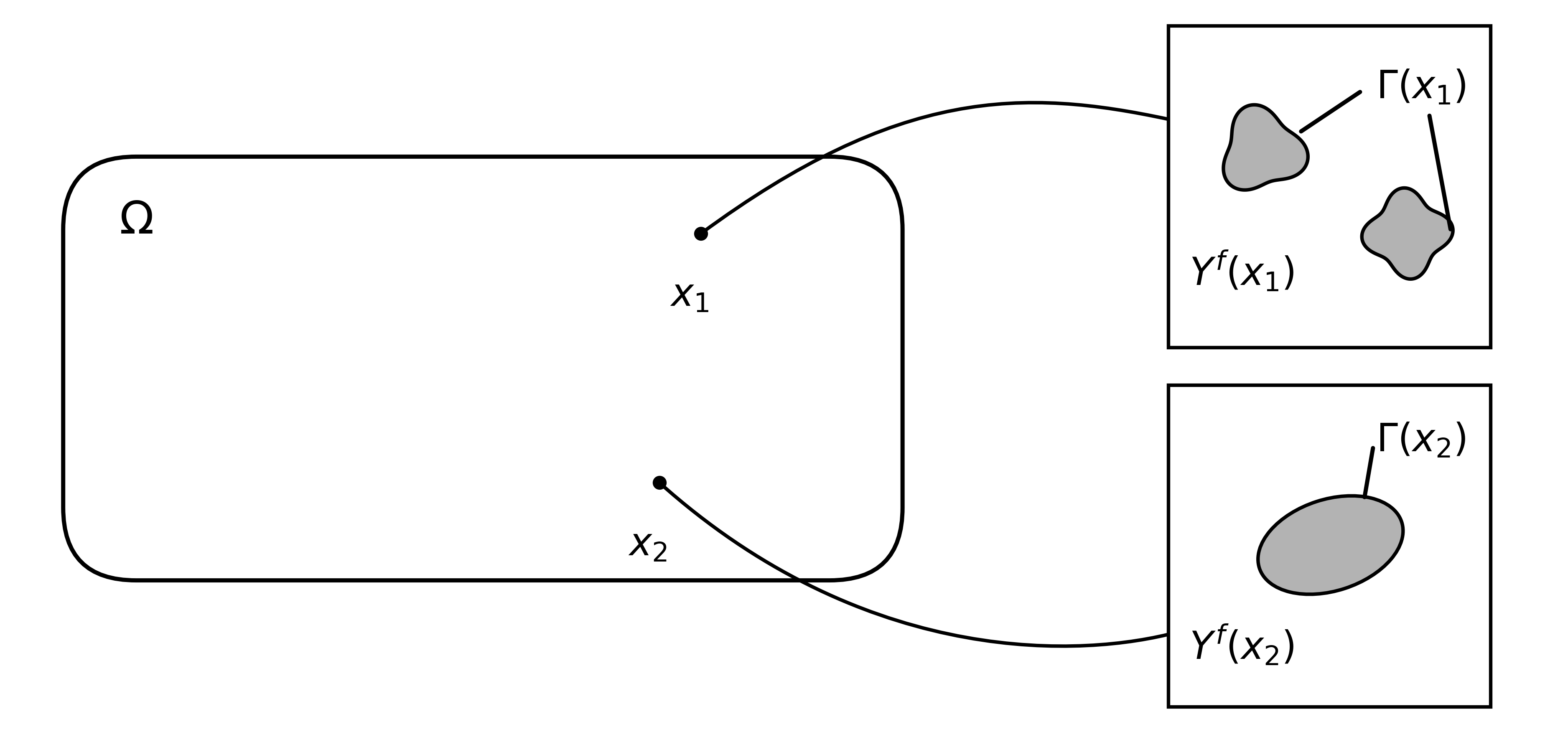}
 \caption{Schematic representation of the two-scale geometry with the macroscopic domain $\Omega$ and the different perforated microstructures $Y^f(x_1)$ and $Y^f(x_2)$ associated with points $x_1,x_2\in\Omega$.
 The precise working assumptions for the microscopic geometries are given within section \ref{PuwS}.}
 \label{pic}
\end{figure}









Our initial motivation was to understand the behaviour of colloidal particles in heterogeneous soils \cite{Krehel,Icardi, Chen}, but very  similar settings appear also in other scenarios like self-healing concrete \cite{Etelvina}, salt crystallization \cite{Bologna}, 
drug-delivery systems
\cite{drug, Giulia}, or, in general,  porous filters 
\cite{Bruna_JFM}. The main objective of the present study is to describe numerically the motion of the freely moving microscopic interfaces using an equation for displacing the level sets.
Our key contributions are two fold:
\begin{itemize}
\item We extend our earlier mathematical analysis from the 2D case to arbitrary dimensions while also allowing for more complex, non-homogeneously distributed microstructures.
In addition, we provide a result (Theorem \ref{thm.uniqueness}) clarifying the  uniqueness of solutions to our model  which was missing in the previous works;
\item Our implementation of the FEM approximation can easily handle a large variations of 2D microstructures (with various varying inner cores). In particular, it allows to investigate clogging mechanisms. 
\end{itemize}

The work is structured as follows: The model we have in mind is described in section \ref{PuwS}. Section \ref{analysis} contains a brief discussion of the model's solvability in terms of suitable weak solutions.  The computability of this class of solutions is the subject of section \ref{numerics}. In this section, we also illustrate the expected solution behavior and emphasize the effect of clogging. We summarize our findings in section \ref{conclusion}. We indicate here also a couple of possible directions for further research in the same context.

\section{Setup, model, and solution concept}\label{PuwS}

We consider a macroscopic domain $\Omega\subset \mathbb{R}^n$, $n=1,2,3$, and use as the macroscopic spatial variable $x\in\Omega$.
We observe the physico-chemical processes taking place within $\Omega$ during the interval $S_T=(0,T)$ with $T>0$.
As time variable we use $t\in S_T$. 
In the framework of this manuscript, the macroscopic domain is assumed to be structured such that, at each macroscopic point $x\in \Omega$, the observer can access an underlying scale $Y\subset \mathbb{R}^d$ (referred here as microscopic\footnote{This is precisely how our model is able to account for the heterogeneity of the internal porous structure of the porous material.}).
For simplicity, we consider the standard unit cell as microscopic domains, i.e., cells of type $Y=(0,1)^3$, with $y\in Y$ denoting the microscopic space variable.
Inside such a cell, we consider a solid core (solid skeleton of the underlying material) upon which colloidal particles tend to deposit causing uniform volume expansion of the solid core.

To make this more precise, we fix a $x$-parametrized family of reference solid configuration $\{Y^s(x)\}_{x\in\Omega}$ with $\overline{Y^s(x)}\subset Y:=(0,1)^d$ and with boundaries $\Gamma(x):=\partial Y^s(x)$.
For $\sigma\in[-\sigma^*,\sigma^*]$ we define the offset surfaces
\[
\Gamma(x,\sigma):=\{y\in Y\ :\ \dist(y,\Gamma(x))=\sigma\},
\]
where $\dist(y,\Gamma(x))$ denotes the signed distance of a point $y\in Y$ to $\Gamma(x)$ (negative inside $Y^s(x)$).
The corresponding solid and fluid regions are
\[
Y^s(x,\sigma):=\{y\in Y\ :\ \dist(y,\Gamma(x))<\sigma\},
\qquad
Y^f(x,\sigma):=Y\setminus\overline{Y^s(x,\sigma)}.
\]
Here, we want to emphasize that $Y^s(x)$ may not be connected and may consist of several inclusions whose number and shapes vary significantly with $x$.
We refer the reader to Figure~\ref{pic} for a sketch of the two–scale geometry we have in mind and to Figure~\ref{growth} for a schematic representation of the involved microstructure.
For the precise setup including all the mathematical assumptions on the reference geometry, we refer to \cref{analysis}.

\begin{figure}[h!]
\centering
\includegraphics[scale=.5]{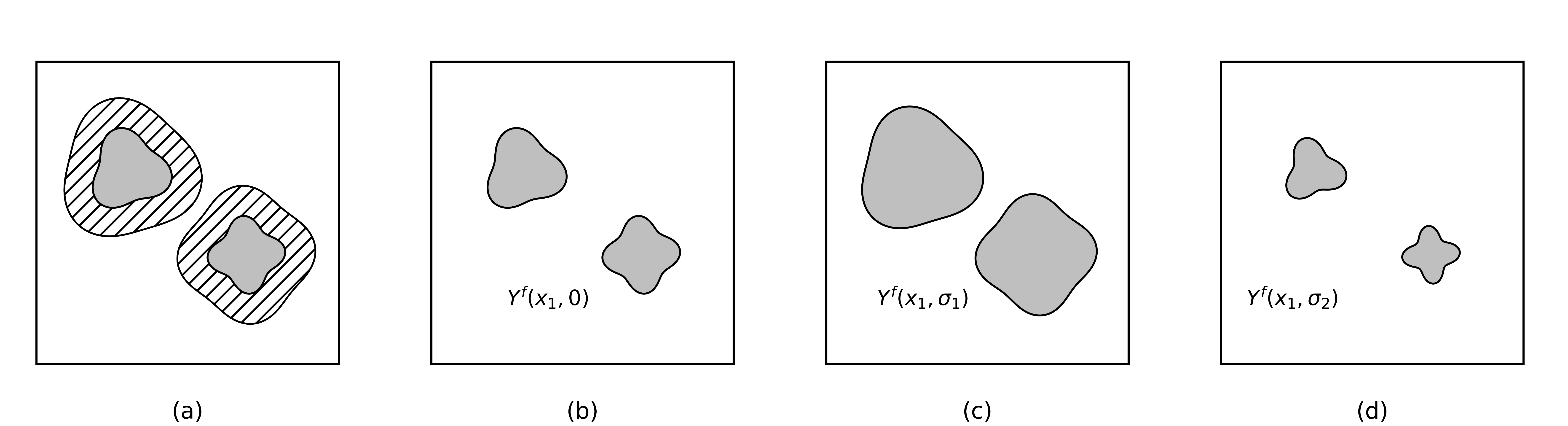}
 \caption{A schematic representation of the growth setup we have in mind. Figure $(a)$ shows the tubular neighborhood of thickness at least $\sigma^*$ for one possible microstructure (linked to $x_1\in\Omega$). Figure $(b)$ depicts the initial setup ($\sigma=0$), $(c)$ the geometry after growth ($\sigma_1>0$), and $(d)$ after shrinkage ($\sigma_2<0$).}
 \label{growth}
\end{figure}

With these notations at hand, we can now describe  the reaction-diffusion system we have in mind, together with its distributed microstructure.
A simpler form of this  model was originally derived in \cite{MC20} and further analyzed in \cite{EdenNikolopoulosMuntean22} and  
\cite{NEM23}. We refer the reader to these works for additional details regarding the problem and its properties.

Already at this early stage of discussion, an important clarification is needed regarding the physiochemical aggregation-coagulation process taking place inside the porous medium and the way we describe it mathematically here. We let $N$ be a given natural number indicating the maximal allowed {\em size} of an aggregate of colloid particles. 
The word {\em size} refers here to the number of primary particles making up the aggregate. The fact that $N<\infty$ indicates that, for any practical situation, there is a maximal size that is reached. For instance, for colloidal species involved in transport and chemical reactions in soils it is not unusual that $N$ is rather small, say around 10-40, a limiting factor for the spread of the particle size distribution being the gravity force; see e.g. \cite{Krehel, Kai,Zech}.

For each $i=1,...,N$, let $u_i\colon S_T\times\Omega\to[0,\infty)$ (we set $u=(u_1,...,u_N)$) denote the molar concentration of aggregates of size $i$  that can be found at point $x\in\Omega$ and time $t\in S$.
Moreover, we take the function $v\colon S_T\times\Omega\to[0,\infty)$ to represent the mass density of absorbed material (mass that is present in the system but currently does not take part in the diffusion and agglomeration processes); this mass can dissolve again allowing colloidal populations to re-enter the pore space. The processes of absorption and desorption is modeled in this context via an Robin-type exchange term agreeing with the linear form of Henry's law (see e.g. \cite{Helmig97,Krehel}) with positive exchange coefficients $a_i$ and $b_i$:
$$
A(x,\sigma) (a_i u_i-b_iv).
$$
Here, the function $A(x,\sigma)=\frac{|\Gamma(x,\sigma)|}{|Y^f(x,\sigma)|}$ is the ratio of the pore space boundary length over the effective volume fluid space and it accounts for the facts that $(a)$ a larger surface area allows for a more efficient exchange and $(b)s$ a smaller fluid area means that any exchange has a stronger effect on the concentration.
For each $i=1,...,N$, $\nicefrac{a_i}{b_i}$ is the equilibrium partition coefficient.
This relationship involving the exchange coefficient $A(x,\sigma)$ is justified via the upscaling procedure performed in \cite{MC20}.

The aggregation processes taking place inside the pore space of the medium is described by a truncated  variant of the {\em Smoluchowski}
 formulation (\cite{Aldous}, \cite{NEM23}) given here by
\begin{equation}\label{smoluchowski}
R_i(u):=\frac{1}{2}\sum_{j+l=i}\gamma_{jl}u_ju_l-u_i\sum_{j=1}^{N-i}\gamma_{ij}u_j,
\end{equation}
with positive coefficients $\gamma_{jl}$.
The index $i$ refers to the $i$th population ($i=1,\dots,N$). 

The first sum in (\ref{smoluchowski}) accounts for the formation via coagulation under the assumption that colloids of size $i$ can be formed when two smaller colloids of sizes $j$ and $l$ with the property $j+l=i$ meet.
In particular, the first sum is zero in $R_1$ (particles of size 1 can not form via coagulation) and $\frac{1}{2}\gamma_{11}u_1^2$ in $R_2$ (particles of size 2 can only form via coagulation of two size-1 colloids).
The second sum accounts for the loss of $i$-sized colloids by coagulating with a different colloid of size $j$ to form a new one of size $i+j\leq N$. 
The second sum is therefore empty in $R_N$. Note that this is conceptually very similar to viewing the overall production process as a chain of second order chemical reactions. 
It is important to observe that in the context of porous media the colloidal populations involve a finite size chain of the cluster, i.e. there will be a population of $N$-mers where $N$ takes the maximum cluster size (see e.g.  \cite{Krehel}, \cite{NEM23}, and references cited therein).

Summarizing the discussion, the diffusion-reaction system for the different $N$ populations of dissolved aggregates is then given via
\begin{subequations}
\begin{empheq}[left=(P_u)\empheqlbrace]{alignat=2}
\partial_t u_i-\dive(D_i(x,\sigma)\nabla u_i) &=R_i(u)-A(x,\sigma)(a_iu_i- b_iv) &\quad& \text{in } S_T\times\Omega, \label{overall-u1_in}\\
-D_i(x,\sigma)\nabla u_i\cdot n&=0&\quad&\text{on } S_T\times\partial\Omega, \label{overall-u2_in}\\
u_i(0)&=u_{i0}&\quad&\text{in } \Omega. \label{overall-u3_in}
\end{empheq}
We refer to this model component as ($P_u$).

The effective diffusion matrix $D_i(x,\sigma)\in{\mathbb{R}}^{d\times d}$ is not just an inhomogeneous coefficient which depends on time and space but rather has to be calculated over the changing microstructure.
More precisely, it is calculated using solutions $w_{k}$ ($k=1,...,d$) of the cell problems
\begin{empheq}[left=(P_w)\empheqlbrace]{alignat=2}
-\Delta_y w_{k} &= 0 &\quad& \text{in } Y^f(x,\sigma), \label{overall-w1_in}\\
-\nabla_y w_{k}\cdot \nu(y) &= e_k\cdot \nu(y) &\quad& \text{on } \partial Y^s(x,\sigma), \label{overall-w2_in}\\
y &\mapsto w_k(x,\sigma,y) & \quad& \text{is $Y$-periodic}\label{overall-w3_in}\\
\int_{Y^f(x,\sigma)}w_k(y)\di{y}&=0\label{overall-w4_in}.
\end{empheq}
We then have 
\[(D_i)_{jk}(x,\sigma)=d_i\phi(x,\sigma)\int_{Y^f(x,\sigma)}(\nabla_y w_{k}(x,\sigma,y)+e_{k})\cdot e_j\di{z},\]
where $\phi(x,\sigma):=|Y^f(x,\sigma)|$ denotes the porosity of the medium, $d_i>0$ are diffusivity constants, and $e_j$ denote the unit vectors.
The structure of both the cell problem and of the effective diffusivity are derived in \cite{MC20} using asymptotic homogenization techniques adapted to the case of  locally-periodic media.
We also point out that with homogenous fluid phase and without growth/shrinkage on the microstructure (so when $\sigma=0$), this structure of the $D_i$ is the classical homogenization result for diffusion in porous media.

The evolution of $v$ is governed by an ODE parametrized in $x\in\Omega$ via
\begin{empheq}[left=(P_v)\empheqlbrace]{alignat=2}
\partial_t v &= \sum_{i=1}^N (a_i u_i -  b_i v) &\quad& \text{in } S_T\times\Omega, \label{overall-v1_in}\\
v(0) &= v_0 &\quad& \text{in } \Omega. \label{overall-v2_in}
\end{empheq}
Finally, we describe the growth/shrinkage of the interface $\Gamma(x,\sigma)$ via the dynamics of the interface offset $\sigma$.
The main assumption is that this dynamics is proportional to the overall mass deposition/dissolution rate.
This leads to the following kinematic law for the interface offset:
\begin{empheq}[left=(P_\sigma)\empheqlbrace]{alignat=2}
\partial_t\sigma
&=\alpha_v \sum_{i=1}^N (a_i u_i -  b_i v)
&\quad& \text{in } S_T\times\Omega \label{overall-S1_in},\\
\sigma(0)&=0&\quad& \text{in } \Omega\label{overall-S2_in}
\end{empheq}
\end{subequations}
where $\alpha_v>0$ is a proportionality coefficient.
This is a direct modification from growing balls in \cite{EdenNikolopoulosMuntean22}, where the same equation was posed for radial evolution, to more general shapes of solid inclusions.

In summary, the overall problem we are considering in this work consists of subproblems $(P_u)$ (\cref{overall-u1_in,overall-u2_in,overall-u3_in}), $(P_w)$ (\cref{overall-w1_in,overall-w2_in,overall-w3_in,overall-w4_in}), $(P_v)$ (\cref{overall-v1_in,overall-v2_in}), and $(P_\sigma)$ (\cref{overall-S1_in,overall-S2_in}).

\subsection{Assumptions and solution concept}
We start by formulating the precise mathematical assumptions specifically in regards to the setup of the microgeometries.

\begin{enumerate}
    \item[\textbf{(A1)}]\label{as:geometry}
        \textit{The setup of the geometry:} Let $\Omega\subset\R^n$ be a bounded Lipschitz domain and $S_T=(0,T)$ for some $T>0$ be the time interval.
        For all $x\in\Omega$, let $Y^s(x)\subset Y=(0,1)^d$ be the initial solid inclusions, $Y^f(x):=Y\setminus Y^s(x)$ the initial micro fluid space, and $\Gamma(x)=\partial Y^s(x)$ the internal interface.
        We assume:
        \begin{itemize}
            \item[$(i)$] For every $x\in\Omega$, $Y^s(x)$ is the union of $k(x)\in\N$ pairwise-disjunct $C^{3}$-domains $Y_j^s(x)$, i.e. $Y^s(x)=\bigcup_{i=1}^{k(x)}Y_j^s(x)$.
            \item[$(ii)$] The solid components are uniformly separated from each other and from the boundary of $Y=(0,1)^d$, i.e., there is $\delta>0$ such that
            \[
            \inf_{x\in\Omega}\min\{\min_{i\neq j}d(Y_i^s(x),Y_j^s(x)),\min_{i}d(Y_i(x);\partial Y)\}\geq \delta.
            \]
            \item[$(iii)$] We assume that there is a uniform tubular neighborhood size $\sigma^*>0$ where
            \begin{multline*}
            \hspace{2cm}\sigma^*:
            =\frac{1}{2}\inf_{x\in\Omega}\sup\{r>0 \ |\ \forall y\in Y \ \text{with}\ d(y,\Gamma(x))<r,\\ \ \exists! z\in \Gamma(x) \ \text{s.t} \ d(y,z)=d(y,\Gamma(x))\}.
            \end{multline*}
            In the terminology of geometric measure theory (e.g., \cite{Federer1969}), the family $\{\Gamma(x)\}_{x\in\Omega}$ has uniformly positive \textit{reach} inside $Y$.
        \end{itemize}
    \item[\textbf{(A2)}]\label{as:coefficients}
        \textit{Coefficients:} Let $N>0$ and $a_i, b_i,d_i,\gamma_{i,j},\alpha_v>0$ for all $i,j=1,...,N$.
    \item[\textbf{(A3)}]\label{as:initial-conditions}
        \textit{Initial conditions:} We assume $u_0=(u_{10},...,u_{iN})\in L^\infty(\Omega)^N$ and $v\in L^\infty(\Omega)$ with $u_{i0}\ge0$ ($i=1,...,N$) and $v\ge0$ almost everywhere in $\Omega$. 
\end{enumerate}
Please note that Assumptions $A1.(i)+(ii)$ imply $\overline{Y^s(x)}\subset Y$ as well as  $k^*:=\sup_{x\in\Omega}k(x)<\infty$.
It is also clear that the constraints in Assumption $(A1)$ enforce a limit on the size of the interface $\Gamma(x)$ and one can show that $|\Gamma(x)|\le 2^{d-2}\left(\min\{\sigma^*,\nicefrac\delta4\}\right)^{-1}$.\footnote{However, this particular bound is not important for our setting, we only need the fact that there is an upper bound.}

Now, using the signed distance function, we set
\[
Y^s(x,\sigma):=\{y\in Y\ :\ d_x(y)<\sigma\},\quad Y^f(x,\sigma)=Y\setminus Y^s(x,\sigma),\quad\Gamma(x,\sigma)=\partial Y^s(x,\sigma)
\]
for all $\sigma\in[-\sigma^*,\sigma^*]$.
Please note that with Assumption $A1.(iii)$, $Y^s(x,\sigma)$ still has $k(x)$ pairwise disjunct components and still satisfies $\overline{Y^s(x,\sigma)}\subset Y$.
In addition, it also ensures that there is $\lambda\in(0,\nicefrac12)$ such that $\lambda\le|Y^s(x,\sigma)|\le1-\lambda$ for all $x\in\Omega$ and $\sigma\in[-\sigma^*,\sigma^*]$.
Similarly, $\lambda\le|Y^f(x,\sigma)|\le1-\lambda$.
We also have the characterization
\[
\Gamma(x,\sigma):=\partial Y^s(x,\sigma)=\{d_x=\sigma\}
\]
and there are $\gamma_1,\gamma_2>0$ such that $\gamma_1\le|\Gamma(x,\sigma)|\le\gamma_2$.

We set $Q(\sigma)=\bigcup_{x\in\Omega}\{x\}\times Y(x,\sigma)$ with $Q:=Q(0)$ denoting the initial micro-macro geometry.

The following table gives an overview over the non-linear coefficients that are included in our model. 
\begin{table}[h]
\begin{tabular}{|c|c|c|}
\hline
Symbol  						& Quantity 										& Expression	 \\ 
\hline
$D_i(x,\sigma)\in\R^{d\times d}$						& effective diffusivity				& $(D_i)_{jk}=d_i\phi(x,\sigma)\int_{Y^f(x,\sigma)}(\nabla_y w_{k}(x,\sigma,\cdot)+e_{k})\cdot e_j\di{z}$     \\ 
\hline
$\phi(x,\sigma)\in(0,1)$						& porosity					 					& $|Y^f(x,\sigma)|$\\
\hline
$A(x,\sigma)>0$						&specific surface area					 					& $A=\frac{|\Gamma(x,\sigma)|}{|Y^f(x,\sigma)|}$\\
\hline
\end{tabular}
\caption{
Notation for the geometric quantities and their dependence on the offset parameter $\sigma$.
}
\end{table}

For convenience of writing,  we introduce the Bochner space
\[
\mathcal{W}(S_T):=\{u\in L^2(S_T;H^1(\Omega))^N\ :\ \partial_tu\in L^2(S_T\times\Omega)^N\}.
\]
We use the $\#$-index to denote periodic boundary condition as in
\[
H^1_\#(Y^f(x,\sigma))=\{w\in H^1_{loc}(\R^d)\ :w\in H^1(Y^f(x,\sigma)), \ w(y)=w(y+e_k)\ \text{a.e.~for all}\ k=1,...,d\}
\]

\begin{defn}[Weak solutions]\label{def:weak_solution}
A weak solution is a set of functions $(u,v,\sigma)$ with the regularity
\begin{align*}
u\in \mathcal{W}(S_T)\cap L^\infty(\Omega\times S_N))^N,\quad \sigma\in L^\infty(\Omega;W^{1,\infty}(S_T)),\quad
v\in L^\infty(\Omega;W^{1,\infty}(S_T))
\end{align*}
that satisfies the following conditions:
\begin{itemize}
    \item[$(i)$] It holds $-\sigma^*\le\sigma(x,t)\le\sigma^*$.
    \item[$(ii)$] The diffusivity of the $i$-nomeres $D_i\in L^\infty(\Omega\times S_T)^{d\times d}$ is given by ($j,k=1,...,d$)
    \[
    (D_i(x,\sigma(t,x))_{jk}=d_i\phi(x,\sigma(t,x))\int_{Y^f(x,\sigma(t,x))}(\nabla w_{k}(x,\sigma(x,t),y)+e_{k})\cdot e_j\di{z}.
    \]
    Here, for any $\sigma\in[-\sigma^*,\sigma^*]$ and $k=1,...,d$ the cell functions $w_{k}(x,\sigma,\cdot)\in H^1_\#(Y(x,\sigma))$ are the unique zero-average solutions to the weak form
    \[
    \int_{Y^f(x,\sigma)}\nabla w_k(x,\sigma,y)\cdot\nabla\xi(y)\di{y}=\int_{\Gamma(x,\sigma)}\xi(y)e_k\cdot n\di{\gamma}
    \]
    for all $\xi\in H^1_\#(Y(x,\sigma))$.
    \item[$(iii)$] The weak forms
    \begin{align*}
    \langle\partial_t u_i,\phi\rangle_{H^1(\Omega)^*}-(D_i(\sigma)\nabla u_i,\nabla\phi)_{L^2(\Omega)} &=(R_i(u),\phi)_{L^2(\Omega)}-(A(\sigma)(a_iu_i- b_iv),\phi)_{L^2(\Omega)}
    \end{align*}
    hold true for all test functions $\phi\in H^1(\Omega)$ and all $1\leq i\leq N$.
    \item[$(iv)$] The equalities
    \begin{align*}
     \partial_tv=\sum_{i=1}^N(a_iu_i- b_iv),\qquad
     \partial_t\sigma=\alpha_v\sum_{i=1}^N(a_iu_i- b_iv).
    \end{align*}
    hold almost everywhere in $\Omega\times S_T$.
\end{itemize}
\end{defn}


\section{Analysis}\label{analysis}
We start by mentioning the main differences to the model investigated and analysed in \cite{EdenNikolopoulosMuntean22}:

\begin{enumerate}
    \item[$(a)$] The changing microgeometry is more complicated. Instead of looking at only growing and shrinking balls, the current setup allows for much more general structures as depicted in, e.g., \cref{pic}.
    In particular, multiple obstacles in a cell are possible now and the shape can depend on the macroscopic variable.
    \item[$(b)$] As a consequence of $(a)$, some macroscopic coefficients like $A(\sigma)$ are different.
    \item[$(c)$] The analysis in \cite{EdenNikolopoulosMuntean22} only looked at the 2D case. Here, we tackle the problem for general $d\in\N$.
\end{enumerate}

In principle, most of the arguments from \cite{EdenNikolopoulosMuntean22} can be fitted to the current setup with only small adjustments.
For that reason, we skip most of the details here; however, for the convenience of the reader we do outline the main strategy of the existence proof and highlight the small but necessary changes to extend  \cite[Theorem 11]{EdenNikolopoulosMuntean22} to our setting in \cref{thm.existence}.
In our earlier work, we did not investigate uniqueness of solutions which we rectify here by establishing a weak-strong uniqueness principle where we show that the existence of a sufficiently regular solution (with the exact regularity dependent on the spatial dimension $d$) automatically guarantees uniqueness of weak solutions (see \cref{thm.uniqueness}).

\begin{lem}[Lipschitz estimates for geometric coefficients]\label{lemma:lipschitz_coefficient}
    There is a Lipschitz constant $C>0$ such that
    \[
    |A(x,\sigma_1)-A(x,\sigma_2)|\leq C|\sigma_1-\sigma_2|
    \]
    for all $x\in\Omega$ and $\sigma_1,\sigma_2\in[-\sigma^*,\sigma^*]$.
\end{lem}
\begin{proof}
    Due to the regularity of the reference domain and the fact that we stay inside the tubular neighbourhood, this is clear for each $x\in\Omega$, i.e., it holds
    \(
    |A(x,\sigma_1)-A(x,\sigma_2)|\leq C(x)|\sigma_1-\sigma_2|
    \)
    for $\sigma_1,\sigma_2\in[-\sigma^*,\sigma^*]$ for some $C(x)\ge0$.
    
    Now, $\delta\le |Y^f(x,\sigma)|\le 1-\delta$ for some $\delta\in(0,\nicefrac12)$.
    For the surface measure, we can estimate 
    \[
    |\Gamma(x,\sigma)|\le(1+\sigma^*\kappa^*(x))^{d-1}|\Gamma(x)|
    \]
    where $\kappa^*(x)$ is the maximal principal curvature of $\Gamma(x,0)$.
    With $\kappa^*(x)\le(\reach(\Gamma(x)))^{-1}\le(2\sigma^*)^{-1}$, we get
    \[
    \sup_{x\in\Omega}|\Gamma(x,\sigma)|\le \left(\frac32\right)^{d-1}\sup_{x\in\Omega}|\Gamma(x)|<\infty.
    \]
    As a consequence, $A(x,\sigma)=\frac{|\Gamma(x,\sigma)|}{|Y^f(x,\sigma)|}$ is also uniformly bounded in $x$ and therefore $\sup_{x\in\Omega}C(x)<\infty$
\end{proof}

Besides these geometric coefficients, the cell problems and therefore the diffusivity also depends directly on $\sigma$.
In the following, we show that they also depend Lipschitz continuously on $\sigma$ and are uniformly coercive:

\begin{lem}\label{lemma:lipschitz.diffusivity}
There is a constant $c_D>0$ such that
$
c_D|\xi|^2\leq D_i(x,\sigma)\xi\cdot\xi
$
for all $\xi\in\R^d$, for all $x\in\R^n$, and for all $\sigma\in[-\sigma^*,\sigma^*]$.
Moreover, it holds $D_i\in L^\infty(\Omega;C^{1,1}([-\sigma^*,\sigma^*];\R^{d\times d}))$.
\end{lem}
\begin{proof}
    A comparable statement for positivity and Lipschitz continuity was shown in \cite[Lemma 4 and Lemma 8]{EdenNikolopoulosMuntean22} for growing/shrinking ball inclusions in 2D.
    The same approach (transforming to the reference geometry for $\sigma=0$) can be utilized here as well.
    While the calculations are a bit tedious, the argument is rather straightforward; we point to \cite[Lemma 3 and Lemma 7]{EdenFreudenbergMuntean2025} where the same is shown for a heat conductivity problem.

    The only thing eventually missing is the additional $x$-dependency of the diffusivities, which are purely a consequence of the initial geometric setup.
    In the context of this work, Assumption~\eqref{as:coefficients} ensures that nothing goes wrong (i.e. local degeneracies are prohibited to happen).
\end{proof}


\begin{thm}[Existence]\label{thm.existence}
There is a time interval $S_T=(0,T)$ and a local-in-time weak solution 
\[
(u,v,\sigma)\in \mathcal{W}(S_T)\cap L^\infty(\Omega\times S_T)\times L^\infty(\Omega;W^{1,\infty}(S_T))\times L^\infty(\Omega;W^{1,\infty}(S_T))
\]
solving the Problem in the sense of \cref{def:weak_solution}.
\end{thm}
\begin{proof}
The main strategy from the existence proof \cite[Theorem 13]{EdenNikolopoulosMuntean22} can mostly be transferred to the more general scenario considered here.
For the convenience of the reader, we shortly outline the solution strategy and highlight the changes in the analysis.

First, in the exact same way as in \cite{EdenNikolopoulosMuntean22}, we can reformulate the coupled problem as a semi-linear, parabolic PDE system:

\textbf{Step (a):} The ODE's for $v$ and $\sigma$ can be explicitly solved (setting $b = \sum_{i=1}^N  b_i$):
\begin{align}
L_{I}(u)(t,x):=v(t, x) &= e^{-bt} \left( v_0(x) + \sum_{i=1}^N a_i \int_0^t e^{b\tau} u_i(\tau, x)\di{\tau} \right),\\
L_{II}(u)(t,x):=\sigma(t, x) &=\alpha_v \sum_{i=1}^N \int_0^t \left( a_i u_i(\tau, x) -  b_i v(\tau, x) \right) \di{\tau}.
\end{align}

\textbf{Step (b):} Looking at the cell problems, we notice that the weak formulation 
\begin{align*}
(\nabla_y w_k,\nabla_y\xi)_{L^2(Y^f(x,\sigma))}&=\int_{\Gamma(x,\sigma)}e_k\cdot n\xi\di{\sigma},\quad (\xi\in H^1_\#(Y),\ k=1,..,d)
\end{align*}
gives us unique, zero-average solutions $w_k\in H_\#^1(Y(x,\sigma))$ for every given $\sigma\in[-\sigma^*,\sigma^*]$ and $x\in\Omega$.
We introduce the corresponding solution operator via
\[
L\colon\Omega\times[-\sigma^*,\sigma^*]\to H_\#^1(Y(x,\sigma))^N,\quad L(x,\sigma,\cdot)=(w_{1},w_{2},...,w_{n})(x,\sigma,\cdot)
\]
and get
\[
L_{III}(u)(t,x):=L(x,L_{II}(u(t,x))).
\]

\vspace{1em}
\textbf{Step (c):} Putting everything together, we can rewrite the parabolic problem into
\begin{multline*}
\langle\partial_t u_i,\phi\rangle_{H^1(\Omega)^*}-(D_i(x,L_{III}(u))\nabla u_i,\nabla\phi)_{L^2(\Omega)}\\ =(R_i(u),\phi)_{L^2(\Omega)}-(A(x,L_{III}(u))(a_iu_i- b_iL_{I}(u)),\phi)_{L^2(\Omega)}
\end{multline*}
which is a semi-linear parabolic PDE for which an extensive solution theory is available using fixed-point arguments as long as $D, R_i, A, L_{I},L_{II},L_{III}$ are well enough behaved.

This is the case as long as it can be ensured that $\sigma\in[-\sigma^*,\sigma^*]$.
This has been shown in \cite{EdenNikolopoulosMuntean22} for radial inclusions (in 2D) and in \cite{EdenMuntean2024} for general $C^3$-inclusions.
With the regularity results from \cref{lemma:lipschitz_coefficient} and \cref{lemma:lipschitz.diffusivity}, this can be transferred to our setup.
To make sure that $\sigma$ stays bounded, $L^\infty$-estimates are needed for the $u_i$; and those are available via maximum principle based on the structure of the right hand side (cf.~\cite[Lemma 8]{EdenNikolopoulosMuntean22}).
Although this was only done for 2D domains, there are no dimensional consideration and the argument can be lifted directly to any space dimension.
From here, Schauder's fixed-point theorem can be used to show the existence of at least one weak solution.
\end{proof}

\begin{rem}
    Let $T^*$ be the maximal existence interval. Then, either $T^*= \infty$ or $\sup_{x}|\sigma(x,T^*)|=\sigma^*$.
    In other words, as long as the offset limits are not reached, the solution can be extended further in time.
\end{rem}

\begin{thm}[Weak-strong uniqueness]\label{thm.uniqueness}
    If there is a weak solution $(u,v,w,\sigma)$ with the additional regularity $u\in L^2(S_T;W^{1,p}(\Omega))^N$ for some $p>d$.
    Then, this is the unique weak solution.
\end{thm}
\begin{proof}
We assume that we have two sets of solutions $(u^{(j)},v^{(j)},w^{(j)},\sigma^{(j)})$ ($j=1,2$) and denote their difference by $(\bar{u},\bar{v},\bar{w}_k,\bar{\sigma})$.
The surface concentrations are given as
\[
v^j(t,x)=e^{-bt}\left(v_0(x)+\sum_{i=1}^Na_i\int_0^te^{b\tau}u_i^j(\tau,x)\di{\tau}\right)
\]
and they are related to the offset parameter $\sigma^{(j)}$ via $\sigma^{(j)}=\alpha_v (v^j-v_0)$.
Therefore, 
\[
\bar{\sigma}(t,x):=\sigma^{(1)}(t,x)-\sigma^{(2)}(t,x)=\alpha_v\sum_{i=1}^N a_i\int_0^te^{b(\tau-t)}\bar{u}_i(\tau,x)\di{\tau}
\]
and
\begin{equation}\label{eq:bar_sigma_estimate}
|\bar{\sigma}(t,x)|\leq\alpha_v\sum_{i=1}^N a_i\int_0^t|\bar{u}_i(\tau,x)|\di{\tau}.
\end{equation}
Now, the weak form of the reaction diffusion problems governing the concentrations $u_i^{(j)}$ is given by
\begin{multline*}
\int_\Omega\partial_t\bar{u}_i\varphi\di{x}+\int_\Omega D_i(\sigma^{(1)})\nabla\bar{u}_i\cdot\nabla\varphi\di{x}
+\int_\Omega\left(D_i(\sigma^{(1)})-D_i(\sigma^{(2)})\right)\nabla u_i^{(2)}\cdot\nabla\varphi\di{x}\\
=\int_\Omega (R_i(u^{(1)})-R_i(u^{(2)}))\varphi\di{x}
+\int_\Omega\left[A(\sigma^1)(a_iu_i^{(1)}- b_iv^{(1)})-A(\sigma^{(2)})(a_iu_i^{(2)}- b_iv^{(2)})\right]\varphi\di{x}.
\end{multline*}
Choosing $\varphi=\bar{u}_i$, integrating over $(0,t)$, and using the uniform positivity of $D_i$, we estimate
\begin{multline}\label{eq:initial_estimate}
\frac{1}{2}\|\bar{u}_i(t)\|_{L^2(\Omega)}^2+c_D\|\nabla\bar{u}_i\|_{L^2(S_T\times\Omega)}^2\\
\leq\underbrace{\int_0^t\int_\Omega\left|\left(D_i(\sigma^1)-D_i(\sigma^2)\right)\nabla u_i^{(2)}\cdot\nabla\bar{u}_i\right|\di{x}\di{\tau}}_{I_1}
+\underbrace{\int_0^t\int_\Omega (R_i(u^1)-R_i(u^2))\bar{u}_i\di{x}\di{\tau}}_{I_2}\\
+\underbrace{\int_0^t\int_\Omega\left[A(\sigma^1)-A(\sigma^2)\right](a_iu_i^{(1)}- b_iv^{(1)})\bar{u}_i\di{x}\di{\tau}}_{I_3}
+\underbrace{\int_0^t\int_\Omega A(\sigma^2)(a_i\bar{u}_i- b_i\bar{v})\bar{u}_i\di{x}\di{\tau}}_{I_4}.
\end{multline}
For the $I_1$-term, we use \cref{lemma:lipschitz.diffusivity}, \cref{eq:bar_sigma_estimate}, and Fubini to estimate via Hölder's inequality:
\[
I_1
\leq L_D\alpha\sum_{k=1}^N a_k\int_0^t\|\bar{u}_k\|_{L^1(0,\tau;L^{q}(\Omega))}\|\nabla u_i^{(2)}\|_{L^{p}(\Omega)}\|\nabla\bar{u}_i\|_{L^2(\Omega)}\di{\tau}
\]
with $q=\frac{2p}{p-2}$ and $p>d$ as in the assumption of the theorem.
We note that, with this choice, $q< q^*:=\frac{2d}{d-2}$ (with $q^*=\infty$ when $d=2$) and utilize the continuous and compact embedding from $H^1(\Omega)$ into $L^q(\Omega)$ for all $q<q^*$.
We now use Ehrling's lemma (note that $H^1(\Omega)$ is compactly embedded into $L^q(\Omega)$ since $q<q*$): for all $\delta>0$ there is $C_{\delta}>0$ such that
\[
\|u\|_{L^q(\Omega)}\leq C_{\delta}\|u\|_{L^2(\Omega)}+\frac{\delta}{2} \|\nabla u\|_{L^2(\Omega)}\quad  (u \in H^1(\Omega)).
\]
This implies
\[
I_1
\leq L_D\alpha\sum_{k=1}^N a_k\int_0^t\left(C_{\delta}\|\bar{u}_k\|_{L^1(0,\tau;L^{2}(\Omega))}+\frac{\delta}{2}\|\nabla\bar{u}_k\|_{L^1(0,\tau;L^{2}(\Omega))}\right)\|\nabla u_i^{(2)}\|_{L^{p}(\Omega)}\|\nabla\bar{u}_i\|_{L^2(\Omega)}\di{\tau}.
\]
Applying Hölder's inequality once more, we further estimate
\begin{equation}\label{eq:estimate_I1}
I_1
\leq L_D\alpha\sqrt{t}\|u_i^{(2)}\|_{L^2((0,t);W^{1,p}(\Omega))}\|\nabla \bar{u}_i\|_{L^2((0,t)\times\Omega)}\sum_{k=1}^N a_k\left(C_\delta\|\bar{u}_k\|_{L^{2}((0,t)\times\Omega)}+\delta\|\nabla\bar{u}_k\|_{L^{2}((0,t)\times\Omega))}\right).
\end{equation}
%
%
The reaction terms $R_i$ are locally Lipschitz continuous, $u_i^{(1)}$ and $u_i^{(2)}$ are essentially bounded (\cref{thm.existence}), and $A$ is uniformly bounded:
\begin{equation*}
I_2+I_4
\leq (L_R+Ca_i)\|\bar{u}_i\|_{L^2((0,t)\times\Omega)}^2+C b_i\int_0^t\int_\Omega|\bar{v}||\bar{u}_i|\di{x}\di{\tau}.
\end{equation*}
The second term on the right-hand side can be estimated the same as $I_1$ considering that $\bar{\sigma}=\alpha\bar{v}$, which leads to
\begin{equation}\label{eq:estimate_I2I4}
I_2+I_4
\leq (L_R+C a_i)\|\bar{u}_i\|_{L^2((0,t)\times\Omega)}^2+C b_i\sqrt{t}\|\bar{u}_i\|_{L^2((0,t)\times\Omega)}\sum_{k=1}^N a_k\|\bar{u}_k\|_{L^2((0,t)\times\Omega)}.
\end{equation}
For the remaining term, $I_3$, we use the Lipschitz continuity of $A$ (\cref{lemma:lipschitz_coefficient}) and \cref{eq:bar_sigma_estimate}:
\begin{equation}\label{eq:estimate_I3}
I_3
\leq L_A\alpha\sqrt{t}\left(a_i\|u_i^{(1)}\|_{L^\infty((0,t)\times\Omega))}+ b_i\|v^{(1)}\|_{L^\infty((0,t)\times\Omega)}\right)\|\bar{u}_i\|_{L^2((0,t)\times\Omega)}\sum_{k=1}^N a_k\|\bar{u}_k\|_{L^2((0,t)\times\Omega)}
\end{equation}
%
Using a generic constant $C$, we thereby get
%
%
\begin{multline*}
\|\bar{u}_i(t)\|_{L^2(\Omega)}^2+\|\nabla\bar{u}_i\|_{L^2((0,t)\times\Omega)}^2\\
\leq C\|\bar{u}_i\|_{L^2((0,t)\times\Omega)}^2
+C\bigg(\|\nabla\bar{u}_i\|_{L^2((0,t)\times\Omega)}\sum_{k=1}^N\left(C_\delta\|\bar{u}_k\|_{L^{2}((0,t)\times\Omega)}+\delta\|\nabla\bar{u}_k\|_{L^{2}((0,t)\times\Omega))}\right)\\
+\|\bar{u}_i\|_{L^2((0,t)\times\Omega)}\sum_{k=1}^N\|\bar{u}_k\|_{L^2((0,t)\times\Omega)}\bigg).
\end{multline*}
With Cauchy-Schwarz's and Young's inequalities, we can simplify this further 
\begin{equation*}
\|\bar{u}_i(t)\|_{L^2(\Omega)}^2+\|\nabla\bar{u}_i\|_{L^2((0,t)\times\Omega)}^2
\leq C\left(\|\bar{u}_i\|_{L^2((0,t)\times\Omega)}^2
+NC_\delta\|\bar{u}\|^2_{L^{2}((0,t)\times\Omega)^N}+N\delta\|\nabla\bar{u}\|^2_{L^{2}((0,t)\times\Omega)^N}\right).
\end{equation*}
Now, summing over $i=1,...,N$, taking the supremum over $S_T$, and subsuming the gradient term on the left hand side by choosing $\delta$ to be sufficiently small, we arrive at
\[
\|\bar{u}(t)\|_{L^2(\Omega)^N}^2+\|\nabla\bar{u}\|_{L^2(S_T\times\Omega)^N}^2\\
\leq C\|\bar{u}\|^2_{L^2(S_T\times\Omega)^N}
\]
which, via Grönwall's inequality, implies $\bar{u}=0$ almost everywhere in $S_T\times\Omega$.
\end{proof}

\begin{rem}
    In general, one cannot expect weak solutions to have this higher gradient integrability without further regularity assumptions on the data.
    In dimension $d=2$, however, it suffices to show $\nabla u\in L^2(S_T;L^{2+\e}(\Omega))$ for some arbitrarily small $\e>0$.
    This is guaranteed via Meyers-type estimates \cite{meyers_lp-estimate_1963} for parabolic systems (as in, e.g., \cite{tian_lorentz_2017}).

    For $d\ge3$, the required regularity can be established via maximal parabolic regularity provided that the initial condition $u_0=(u_{10},...,u_{N0})$ satisfies $u_0\in W^{1,p}(\Omega)^N$ together with the compatibility conditions $\nabla u_{i0}\cdot n=0$ almost everywhere on $\partial\Omega$; it then follows that $u\in L^2(S_T;W^{1,p}(\Omega))$ (\cite{amann_linear_1995}).
    Note that the forcing term belongs to $L^2(S_T;L^P(\Omega))$ since the solutions are uniformly bounded.
\end{rem}


\section{Numerical simulation of the two-scale quasilinear  problem 
}\label{numerics}
In this section, we illustrate numerically the behavior of the weak solution to our model for selected model parameters and particular choices of two-dimensional microscopic and macroscopic geometries. Essentially, we illustrate that our model is able to capture clogging effects due to the evolution of the moving boundaries of the inner microscopic cores. Other choices of two-dimensional geometries can be as well considered, if a certain real-world application would require it. More implementation work is though needed if the involved geometries are three-dimensional. Closely related studies which treat two-scale problems that treat space- and/or time-dependent  microstructures are \cite{Omar, EdenFreudenbergMuntean2025}.

To start with we make two important observations  in order to simplify the overall numerical treatment of the problem. The first refers to a change of variable of the kinetic condition that allow us to separate its solution, given now by the Eikonal equation, from the implicit entanglement of the problem while the second allow us to obtain  analytical solutions of the Eikonal equation.


\subsection{Exact solution}
 \subsubsection{Change of time variable}
 Let  $S_a$ be the initial position of
 the moving boundary,  {  applied to all the cells inside the material, assumed to be independent of $x$ while} for $t>0$ its position is given by the points $(y_1, y_2)$ for which we have $y_2=S(x,y_1,t)$ or more generally, as we can see in \eqref{EikSol}, in parametric form $(\alpha_0(s), \beta_0(s))$, $s\in [0,2\pi]$, a $C^1$-curve. We make this choice to keep the presentation simple\footnote{As a natural extension, we can include a dependence on the variable $x$ (e.g. by using a random distribution) and then take $S_a=S_a(x,y)$. However, such a choice would naturally lead to additional computational challenges that we wish to avoid at this stage.}.

  In order to simplify things and not to solve simultaneously the equations for $u$ and $S$, we propose to use an appropriate change of the time variable.

More specifically, we focus on the equation for $S=S(x,y_1,t)$, where
we have that the speed of the moving boundary $\mathcal{V}$ should be proportional to the rate of increase of the deposited colloidal species on the surface of the solid core of a cell.

\begin{eqnarray*}\label{eqS10}
\mathcal{V} = \frac{\partial S}{\partial t}\frac{1}{ \sqrt{1+ \left(\frac{\partial S}{\partial y_1}\right)^2}} =
\alpha_v \frac{\partial v}{\partial t},
 \,\, S(x,y_1,0)=S_a(y_1), \,\, 
\end{eqnarray*}
or
\begin{eqnarray}\label{eqS1}
\frac{\partial S}{\partial t} =
 \sqrt{1+ \left(\frac{\partial S}{\partial y_1}\right)^2}\alpha_v \sum_{i=1}^N\left(a_iu_i- b_iv\right),\\
\label{eqSa1}
 \,\, S(x,y_1,0)=S_a(y_1), \,\, 
\end{eqnarray}

In order to simplify the description, we set ${\partial \sigma}/{\partial \tau}=F_v(x,t)$
 \[F_v(x,t):=\alpha_v \sum_{i=1}^N\left(a_i u_i- b_i v\right)\]
 or

\begin{eqnarray} 
\sigma=\sigma(t)=\int_0^{t} F_v(x,t')dt'. \label{trnsf}
\end{eqnarray}

 Therefore, we can write for $S=S(y_1,\sigma)$,
 \bse
 \label{eqS3}
 \be  \label{eqS3a}
 \frac{\partial S}{\partial\sigma} =
 \sqrt{1+ \left(\frac{\partial S}{\partial y_1}\right)^2},\quad 0<y_1<1,\quad \sigma\geq 0,
 \ee
 \be
\label{eqS3b}
  S(y_1, 0)=S_a(y_1)=0.
\ee
 \ese
This is precisely the form of the Eikonal equation $\left({\partial S }/{\partial \sigma}\right)^2
-\left({\partial S }/{\partial y_1}\right)^2=1$, for $S=S(y_1,\sigma)$.


\subsubsection{Exact solution of the Eikonal equation}\label{SecEikSol}
We will consider the case  that the expanding solid core has a boundary which is a  closed curve  of
the form  $(\alpha(s), \beta(s))=(y_1(s), S(s))$. We consider the curve to be contained into a  unit square cell. 
Essentially, 
this construction corresponds to the following equation:  Find $S=S(y_1,\sigma)$ such that 
\begin{eqnarray}\label{eqSsq}
& \frac{\partial S}{\partial\sigma} =
 \sqrt{1+ \left(\frac{\partial S}{\partial y_1}\right)^2},\quad -1<y_1<1,\quad \sigma\geq 0,\\
&      S(y_1(s),0)=S_0 (y_1(s)),\quad  (y_1(s), \, S_0(s))=(\alpha(s), \beta(s)),  \quad   \sigma(s)=0,\quad     s\in [0.\,2\pi],
\end{eqnarray}
for $\alpha(s), \beta(s)$ given $C^1$ functions. The Cauchy data in this case is ${y_1}_0(s)$, $\sigma_0(s)=0$,
$S_0(s)=S(0,y_1(s))$.

The equation of motion \eqref{eqSsq} can be solved analytically following  the so-called Charpit's method,  as presented e.g. in \cite{HLMO} (see also \cite{Charpit} for a more recent reference). This is a methodology specific for guessing analytic representations of smooth solutions  to selected quasilinear first-order
partial differential equations. Concretely, we set $p:=S_{y_1}$ and $q:=S_{\sigma}$ and then the target equation \eqref{eqSsq}  takes the form $G(p,q)=q^2-p^2-1=0$. Here, we have $q>0$ since $q=\sqrt{1+p^2}$. The relevant corresponding  Charpit's equations are:
 \begin{eqnarray*}\label{eqCe}
& \dot{{y_1}}=\frac{\partial G}{\partial p},\\
& \dot{\sigma}=\frac{\partial G}{\partial q},\\
& \dot S =p\frac{\partial G}{\partial p}+q\frac{\partial G}{\partial q},\\
 & \dot p=- \frac{\partial G}{\partial y_1}-p \frac{\partial G}{\partial S}, \\
& \dot q=- \frac{\partial G}{\partial \sigma}-q \frac{\partial G}{\partial S},
\end{eqnarray*}
 where ``\,\,$\dot{}=\frac{d}{d \xi}$\,". 
 Finally, we get 
 \begin{eqnarray*}
 & \dot {y_1}=G_p,\\
 & \dot \sigma=G_q,\\
& \dot S=pG_p+qG_q,\\
 & \dot p=\dot q=0.
 \end{eqnarray*}
Thus it holds $\dot S=pG_p+qG_q=2q^2-2p^2=2$, while

  \[  p=p_0(\xi) ,\quad q=q_0(\xi).\]
 
 The ray equations are then 
 \begin{eqnarray*}
 y_1={y_1}_0+\xi\left.G_p\right|_0={y_1}_0-2p_0\xi,\\
 \sigma={\sigma}_0+\xi \left. G_q\right|_0=\sigma_0+2q_0\xi,\\
 S=S_0+\xi\left.\left(p G_p +q G_q \right)\right|_0=S_0+2\xi,
 \end{eqnarray*}
  given that
\begin{eqnarray*}
 G_p=-2p,\\
 \left.G_p\right|_0=-2p_0, \\
  G_q=2q,\\
 \left.G_q\right|_0=2q_0.
\end{eqnarray*}   
The initial curve  in parametric form reads 
\[   {(y_1}_0(s), \, S_0(s))=(\alpha_0(s), \beta_0(s)),  \quad   \sigma(s)=0,\quad     s\in [0,\,2\pi].\]  
Also, we have  
\begin{eqnarray*}
q_0^2-p_0^2=1,\\
\frac{d{S_0}}{ds}=\frac{d{{y_1}_0}}{ds}p_0+\frac{d\sigma_0}{ds},
\end{eqnarray*}
or
\begin{eqnarray*}
q_0^2-p_0^2=1,\\
\frac{d{S_0}}{ds}=\frac{d{{y_1}_0}}{ds}p_0.
\end{eqnarray*}
To fix ideas, we look at the case of an initial unit circle.  For this particular situation,  we can write:
 \[ -\sin(s)p_0=\cos(s), \quad p_0=-\cos(s)/\sin(s).\]
 Thus $q_0^2=1+p_0^2=1/\sin^2(s)$, 
 \[q_0= 1/\sin(s),\]  
 and
  \begin{eqnarray*}
&   \sigma=2\xi/\sin(s),\quad\mbox{or}\quad 2\xi=\sigma \sin(s),\\
& y_1=\cos(s)-2\xi (-\cos(s)/\sin(s))=\cos(s)+ (\cos(s)/\sin(s)) (\sigma\sin(s))=\cos(s)[1+\sigma]),\\
& S=\sin(s)+2\xi=\sin(s)+\sigma\sin(s)=\sin(s)[1+\sigma].
 \end{eqnarray*}
 Looking at  the general case, it holds
\begin{eqnarray*}
q_0^2=1+p_0^2,\\
p_0=\frac{S_0'}{{y_1}_0'},
\end{eqnarray*} 
which then gives { for $q_0={S_0}\sigma>0$} that
  \begin{eqnarray*}
   \sigma=2q_0\xi=2\xi \left( \sqrt{1+\left(\frac{S_0'}{{y_1}_0'}\right)^2} \right) ,\\
   2\xi=\sigma \frac{|{y_1}_0'|}{\sqrt{\left({y_1}_0'\right)^2 +\left(S_0'\right)^2}},\\
   \\
 y_1={y_1}_0 - \sigma\frac{|{y_1}_0'|}{{y_1}_0'} \frac{S_0'}{\sqrt{\left({y_1}_0'\right)^2 +\left(S_0'\right)^2}},\\
\\
 S=S_0 + \sigma \frac{|{y_1}_0'|}{\sqrt{\left({y_1}_0'\right)^2 +\left(S_0'\right)^2}},
 \end{eqnarray*}
 
 
 At each time $\sigma$, we obtain a curve in parametric form in terms of $\alpha$ and $\beta$, which is written explicitly as
  \begin{eqnarray}\label{EikSol}
&  
\left(\alpha(\sigma,s),\beta(\sigma,s)\right)=\left(\alpha_0(s)-\sigma \frac{|\alpha_0'(s)|}{\alpha_0'(s)} \frac{\beta_0'(s)}{\sqrt{\left(\alpha_0'(s)\right)^2 +\left(\beta_0'(s)\right)^2}} ,\,\,
\beta_0(s) + \sigma \frac{|\alpha_0'(s)|}{\sqrt{\left(\alpha_0'(s)\right)^2 +\left(\beta_0'(s)\right)^2}}\right),
\quad \quad\\
& s\in [0,\,2\pi].\nonumber
 \end{eqnarray}
 Now, we can  finally  substitute   \eqref{eqS1} and \eqref{eqSa1} with \eqref{trnsf} and \eqref{EikSol} (i.e. the  $(P_S)$ problem), which results in the presence of a moving interface of the form:
 \begin{empheq}{align*}
 \left(\alpha(\sigma,s),\beta(\sigma,s)\right)=\left(\alpha_0(s)+\sigma g_\alpha(s) ,\,\,
\beta_0(s) + \sigma g_\beta(s)\right), \\
\sigma=\sigma(t)=\int_0^{t} F_v(x,t')dt', \quad F_v(x,t)=\alpha_v \sum_{i=1}^N\left(a_i u_i- b_i v)\right),
\end{empheq}
with $g_\alpha(s):=-\frac{|\alpha_0'(s)|}{\alpha_0'(s)} \frac{\beta_0'(s)}{\sqrt{\left(\alpha_0'(s)\right)^2 +\left(\beta_0'(s)\right)^2}}$ and $g_\beta(s):=\frac{|\alpha_0'(s)|}{\sqrt{\left(\alpha_0'(s)\right)^2 +\left(\beta_0'(s)\right)^2}}$.

\subsection{Model equations and geometry}

In this section, we wish to solve numerically the macroscopic model problem
for the $i$-mers species concentration $u_i$ ($i=1,\dots,N$) and the mass density of the absorbed material $v$. The upscaled model equations are posed in a two-dimensional macroscopic environment and have direct access via cell problems to a two-dimensional microstructure environment as we will explain in this section. 
Primarily, we want to observe the basic qualitative behavior of the model towards clogging in a
two dimensional  macroscopic domain around typical geometrical singularities. To fix ideas, we discuss two cases of macroscopic environments:  a cardioid domain and an L-shaped domain.  

To focus our attention on physically relevant parameters choices, we use the setup described in \cite{johnson1995dynamics}; see also \cite{Krehel}, \cite{MC20} for more details.
Essentially, we look at a theoretical model describing the dynamics of colloid
deposition on collector surfaces,  when  both inter-particle and particle-surface electrostatic interactions   are assumed to be negligible.
The numerical range of the used parameters is adapted to the situations that can relate, amongst others, to the immobilization of bio-colloids in soils.

The simulation output we are looking for includes approximated space and time concentration profiles of colloidal populations, spatial distribution of microstructures for given time slices, and estimated amount of deposited colloidal mass. This information helps us detect in {\em a posteriori} way the locations in $\Omega$ where deposition-induced clogging is likely to happen.

Additionally, we want to investigate the effect that  certain choices of  microstructure settings and material geometries have upon the structure of the diffusion tensor (and of the corresponding time evolution of each entry in this tensor).

Recall that our model consists, to begin with, of equation  (\ref{overall-u1_in}) describing the diffusion of $u_i$ in the macroscopic domain $\Omega$.
The effective diffusion tensor has the form  
$$(D_i)_{jk}=d_i\phi(x,\sigma)\int_{Y\setminus \overline{Y}_c(t)}(\nabla w_{k}+e_{k})\cdot e_j\di{z}$$
    for all $i=1,\ldots,N$ and $j,k=1,2$.

 In addition, the length $L(x,t)=|\Gamma(x,t)|$ and area $V(x,t)=|Y^f (x,t)|$ (corresponding surface area and volume in a 3D setting) functions related to the motion of the boundary  are given by 
  \be 
  L=\int_{\mathcal{S}}d {\Gamma },\quad V=\left[1-\int_{Y^f{(t)}}  dy\right],\quad A=\frac{L}{V},
         \label{overall-aux}
 \ee
 while $R_i(u)$ is given by 
\be \label{mod1hb}
R_i(u):=\frac{1}{2}\sum_{j+l=i}\gamma_{jl}u_ju_l-u_i\sum_{j=1}^{N-i}\gamma_{ij}u_j,
\ee
 Moreover, the cell functions $w:=(w_1(x,\sigma,y),w_2(x,\sigma,y))$, assumed to have constant mean, satisfy
 problem $(P_w)$ i.e. equations \eqref{overall-w1_in}, \eqref{overall-w2_in}, \eqref{overall-w3_in},
 
The immobile species $v$ given initially by problem $(P_v)$ i.e. equations 
 \eqref{overall-v1_in}, \eqref{overall-v2_in}, can be tracked by the explicit relation
\begin{eqnarray}\label{alter-v}
	v(t,x)=e^{-bt}\left(v_0(x)+\sum_{i=1}^Na_i\int_0^te^{b\tau}u_i(\tau,x)\di{\tau}\right),
\end{eqnarray}
for a given $v_0(x)$.
 
 Note also that combining equations  \eqref{overall-v1_in},  \eqref{eqS1} and \eqref{trnsf} leads to
 \be \label{mod-sigmav}
{\partial \sigma}/{\partial \tau}=\alpha_v {\partial v}/{\partial \tau}.
\ee
 Consequently, given $\sigma(0)=0$, it results that $\sigma=a_v v(t,x)$.
 

We need to impose boundary conditions to the chosen two-dimensional macroscopic domain $\Omega$. We select non-homogeneous  Robin  boundary conditions as this is the most realistic choice  to simulate inflow through a material interface, namely 
 \be \label{mod1b}
 \frac{\partial u_i}{\partial n}(x_1,x_2,t)+ b_r u_i (x_1,x_2,t)=\left\{\begin{array}{cc}
 u_i^b(x_1,x_2)>0 & t\in [0,t_0],\\
 0         & t>t_0,
 \end{array}\right. \quad (x_1,x_2)\in \partial \Omega.
 \ee
 The parameter $t_0$ represents the upper time limit for which we allow inflow at the prescribed boundary. Depending on the simulation objectives, this time moment  can be identified with the total simulation time.

 
 Finally, we also need initial distributions for all the species $u_i$ and for $v$; they are given by \eqref{overall-u3_in} and \eqref{overall-v2_in}. Additionally,  the solution of our cell problems needs to capture the geometry of the expanding core given by \eqref{EikSol}, starting off from an initial choice of $(\alpha_0(s), \beta_0(s))$, which are required ingredients in the simulations that follow.

  
 

\subsection{Solution strategy and basic simulations}

To approximate numerically the problem posed in \cref{overall-u1_in,overall-u2_in,overall-u3_in,overall-w1_in,overall-w2_in,overall-w3_in,overall-v1_in,overall-v2_in,overall-S1_in,overall-S2_in}, with the additional equations \eqref{eqS1}, \eqref{eqSa1},\eqref{mod1b}, we follow the same solution strategy as proposed in \cite{EdenNikolopoulosMuntean22}. Essentially, we proceed as follows:

We first need to obtain a numerical approximation for the cell problems $(P_w)$, i.e., equations \eqref{overall-w1_in}, \eqref{overall-w2_in}, \eqref{overall-w3_in} and determine the shape of the corresponding cell functions $w_1,w_2$ posed in $\Omega\times S_{\sigma}\times  Y^s(\sigma)$ for $S_{\sigma}=[\sigma_0,\,T_\sigma]$. 

To do this, we also need knowledge of $\partial Y^s(\sigma)$, 
to identify the domain of solution $w_i$.
  More specifically, we proceed for a sequence of domains specified by the  parametric form of the inner boundary 
  $\partial Y^s(\sigma)$ given by equation \eqref{EikSol} for a range of the variable 
   $\sigma\in [\sigma_0,\,T_{\sigma}]$  with $max_{(\sigma, s)}\|S(\sigma,s)< 1/2\|$. In fact we consider 
  a partition of width $\delta \sigma$, $\sigma_0, \sigma_1=\sigma_0+\delta \sigma,\ldots, \sigma_{M}=T_{\sigma}$.
 We choose $T_{\sigma}$ in such a way so that $max_{(T_\sigma, s)} S(T_\sigma,s)=1/2-\epsilon$ for some $\epsilon\ll 1$ i.e. we want $S$ not to touch the outer boundary $\partial Y$ of the cell. 
 Thinking of \cite{Thomas}, we consider that, at the initial time of the overall process,  the interface $\partial Y^s(\sigma)$ is at least of class $C^{1,1}$ and tacitly assume that this boundary regularity does not worsen as time elapses (at least for a short while).
  
   Then since $Y^f(\sigma)= Y\setminus \bar{Y}_c(\sigma)$ is determined as the area contained inside the square cell and outside  the closed curve $S$,  we obtain a sequence of solutions for the cell problem $(P_w)$ for each 
$  Y^f(\sigma_i)$ corresponding to the variable $\sigma_i$ of the partition. 

 We  use a finite element scheme to solve these cell problems. 
    To be precise, we use  a solver in the open source finite element package ''\texttt{FreeFem++}" (see ~\cite{FreeFEM} for details) 
 implemented  to   triangulate the domain $Y^f(\sigma_i)$ and  to solve this specific problem (equations \cref{overall-w1_in,overall-w2_in,overall-w3_in}) with $y$-periodic boundary conditions in $\partial Y$; it works in a similar fashion as applied in \cite{EdenNikolopoulosMuntean22}  \cite{MC20}, \cite{CN18}.  The aforementioned finite element package is chosen  for its flexibility to handle  problems with periodic boundary conditions.


Having the solution $w_k$, $k=1,2$ of the cell problems we can calculate the diffusion tensor   from the relation
$(D_i)_{jk}=d_i\phi(x,\sigma)\int_{Y\setminus \overline{Y}_c(t)}(\nabla w_{k}+e_{k})\cdot e_j\di{z}$
    for all $i=1,\ldots,N$, $j,k=1,2$, and 
we can then tackle the macroscopic problem.  We use the finite element method to solve the  two-dimensional version of
the field equation \eqref{overall-u1_in}, \eqref{mod1hb} and \eqref{alter-v}
with its boundary and initial conditions, \cref{mod1b,overall-u3_in}.
For this purpose, we implement a finite element scheme in the MATLAB finite element  package ''\texttt{PDE Toolbox}"  (see \cite{pdetoolbox} for details)
making the most of its ability to handle 2D geometries and implement a solver for a complex  system of coupled equations. 



\subsection{Cell problems with $y$-periodic boundary}

In this section, we focus on approximating numerically the weak solution to the involved cell problems. To begin with,   we present   in  Figure \ref{FigCellper} the solution  of the cell problem assigned to the cell function   $w_1$ for the case when we have  $y$-periodic boundary conditions at the outer cell boundary for different inner core domains plotted in the panels (a)--(d).  In the panel (a), we take a circle with parametric equation ($(R_c\cos(s),\,R_c\sin(s)))$, $s\in[0,2\pi]$), $R_c=0.2$, in (b) we look at an ellipse ($(R_{a}\cos(s),\,R_b\sin(s)))$, $s\in[0,2\pi]$), $R_a=0.1$, $R_b=0.01$, rotated $150^o$ from the $x-$axis and with its axis ratio being $\frac{R_b}{R_a}=0.1$,  while in (c)  we see   the same ellipse with the same axis ratio but rotated $135^o$ from the $x-$axis. Finally, in the panel (d) we consider the  ellipse ($(R_a\cos(s),\,R_b\sin(s)))$, $s\in[0,2\pi]$) $R_a=0.4$, $R_b=0.2$ with axis ratio $0.5$. 
\begin{figure}[htb]
\begin{center}
\includegraphics[bb= 200 200 450 600, scale=.4]{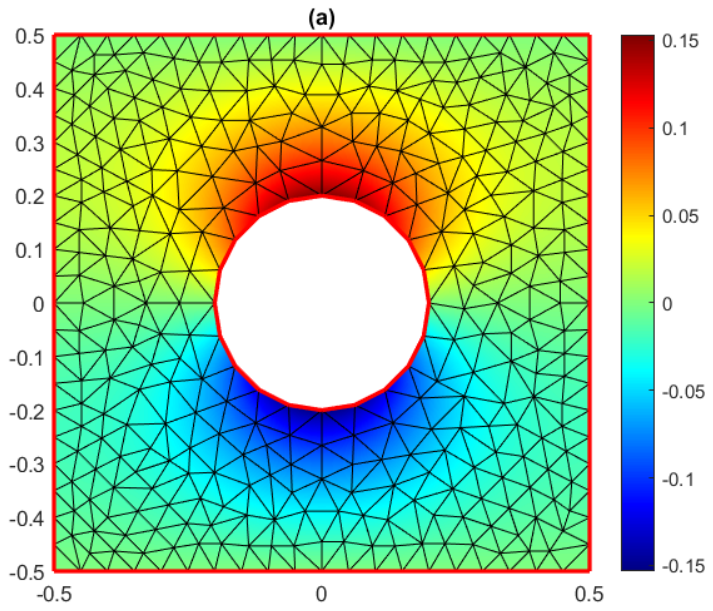}
\includegraphics[bb=0 200 450 600, scale=.4]{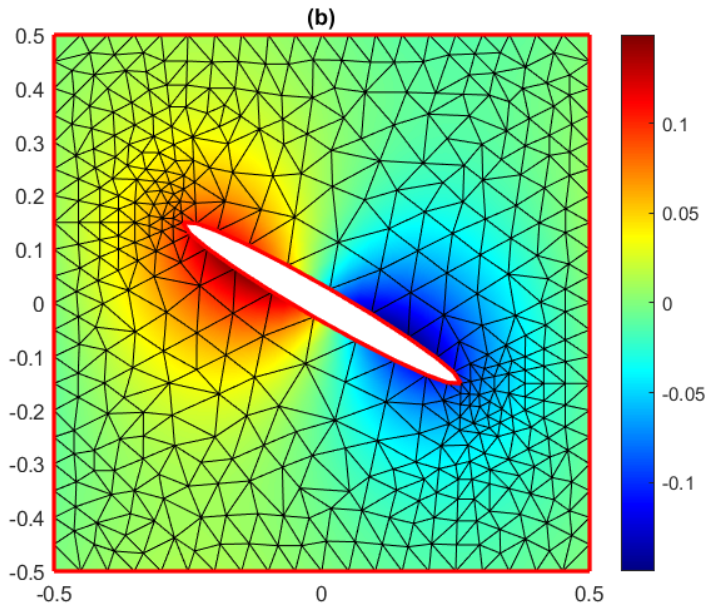}\\
\includegraphics[bb= 200 200 450 600, scale=.4]{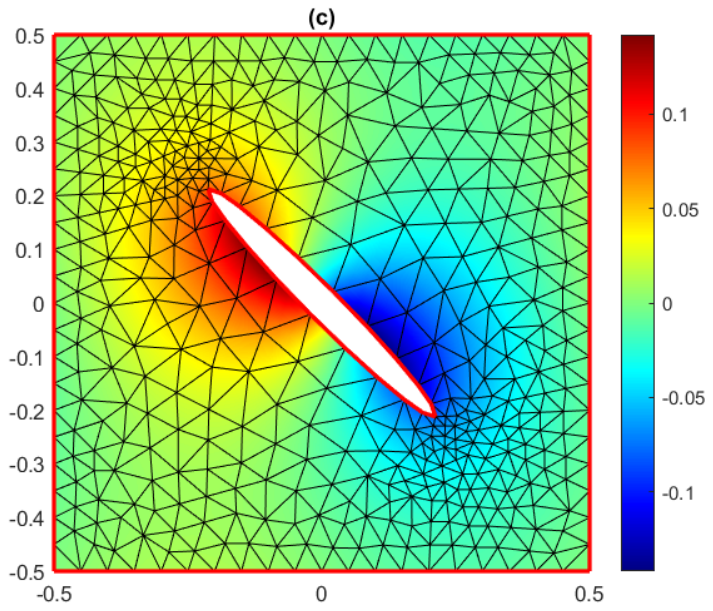}
\includegraphics[bb= 0 200 450 600, scale=.4]{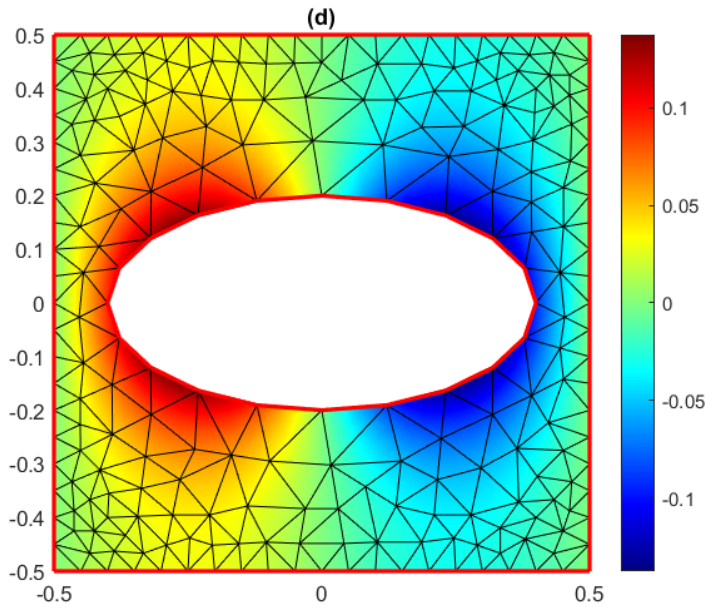}
\end{center}
\caption{\it Numerical solution of the cell problem (equations \cref{overall-w1_in,overall-w2_in,overall-w3_in}) and specifically for $w_1$  with $S$ being: (a) a circle; (b) an ellipse with its long axis forming a $150^o$ angle with the $x-axis$; (c) an ellipse with its long axis forming a $135^o$ angle with the x-axis; (d) an ellipse with axis ratio $0.5$ and with his long axis on the $x-$axis.}\label{FigCellper}
\end{figure}

In the following graphs shown  in  Figure \ref{FigDper}, we present the entries of the diffusion tensor plotted with respect to $\sigma$, of course after having solved previously the corresponding cell problems with $y$-periodic boundary  conditions posed for a similar variation of microscopic core shapes. Specifically, this is done for the case that the initial shape of the core inside the cell is (a) a circle; (b) an ellipse with its long axis forming a $30^o$ angle with the x-axis; 
(c) an ellipse with its long axis forming a $45^o$  
angle with the x-axis; and (d) a bean-shaped curve given by the parametric form $(\alpha(s), \beta(s))$, $s\in[0,\, \pi)$. To this end, we set
$$\alpha(s):=R_c\cos(s)(\cos^3(s)+\sin^3(s)), \beta(s):=R_c\sin(s)(\cos^3(s)+\sin^3(s)).$$ At the initial moment ($\sigma=0$), the ellipse axes are
$R_a(0)=0.01$ and $R_b(0)=0.001$ in the second and third cases, while for the case of the circle, the initial radius is $R_c(0)=0.001$. 
For the bean-shaped curves, the initial parameter $R_c$ is $R_c(0)=0.001$. A partition of $M_\sigma=60$ points is taken in the interval  $[0,T_\sigma]$ for $T_\sigma\, :\,\,\max_{(T_\sigma, s)} S(T_\sigma,s)=1/2-\epsilon$  for $\epsilon=10^{-3}$. 
\begin{figure}[htb]
\begin{center}
\includegraphics[bb= 200 200 450 600, scale=.4]{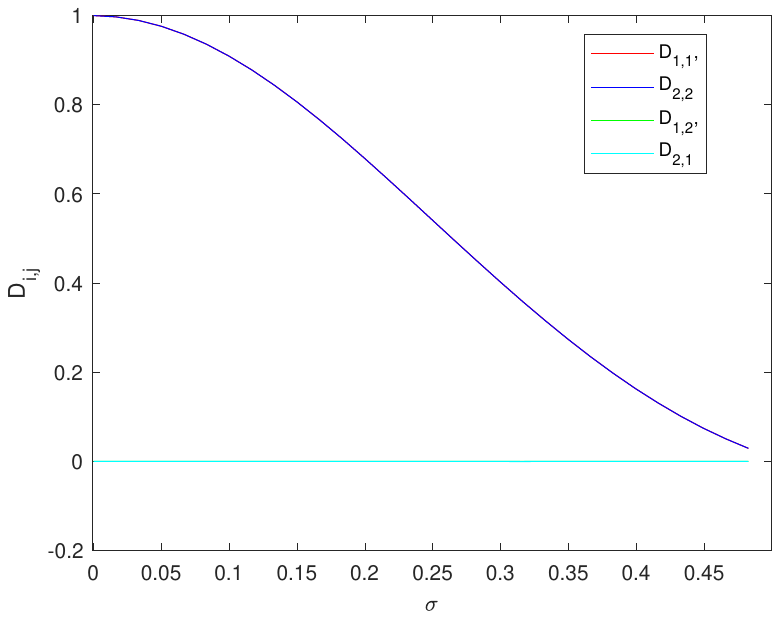}
\includegraphics[bb=0 200 450 600, scale=.4]{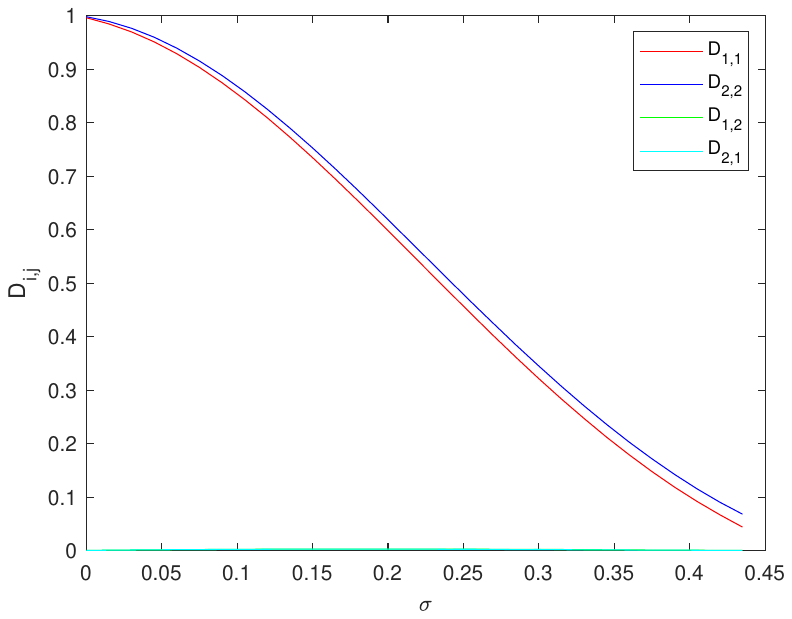}\\
\includegraphics[bb= 200 200 450 600, scale=.4]{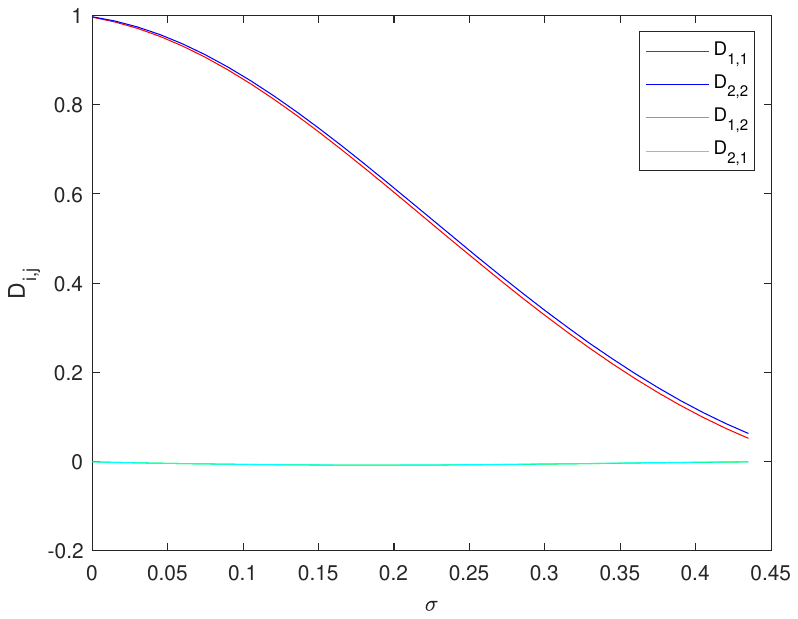}
\includegraphics[bb= 0 200 450 600, scale=.4]{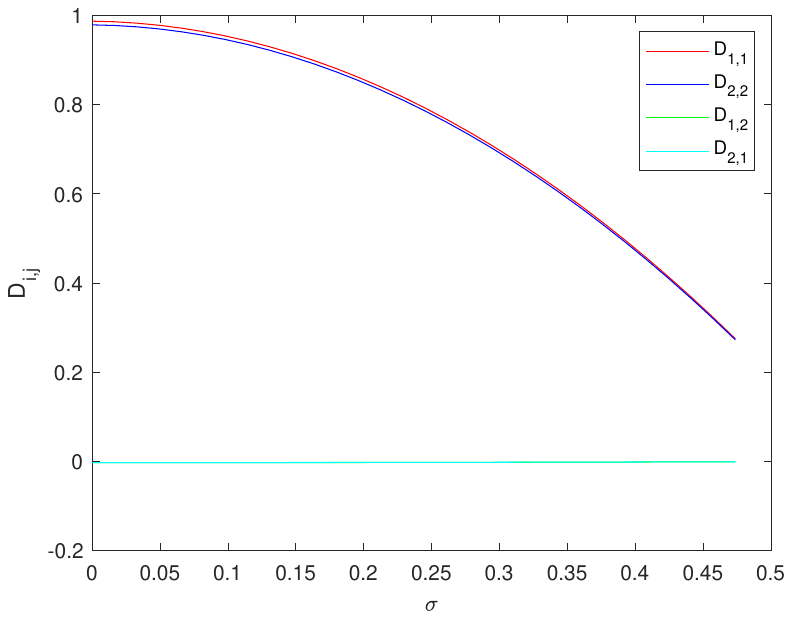}
\end{center}
\caption{\it Entries of the diffusion tensor with respect to the parameter $\sigma$ in the case that the initial shape is (a) a circle (b) an ellipse with its long axis forming a 30$^o$ angle with the x-axis, (c) an ellipse with its long axis forming a 45$^o$ angle with the x-axis and  (d) a bean shaped curve.}\label{FigDper}
\end{figure}

    As we can see, the significant entries of the diffusion tensor are the diagonal ones (related to ${w_1}_x$ and ${w_2}_y$) in all cases. Their values are almost identical, with the exception for the case of the rotated ellipse at a 30$^o$ from the $x-$axis where we capture a small difference between the diagonal entries. That difference may be due to the form of the domain, which in this case is not symmetrical along the axis and the diagonals of the square.
    Moreover, in the first three cases we observe a slightly sigmoidal shape of the curves while in the last case (of a bit more complex bean shape), we have  curves like $1-c \sigma^2$ for some $c>0$. 
\newline
Summarizing our main thoughts here, by making the initial choice of the microstructure geometry via the solution of the Eikonal equation combined with the numerical solution of the cell problems, we can precalculate the diffusion tensor needed for the numerical solution of the macroscopic problem. 


\subsection{Simulations for a cardioid macroscopic domain} 
We now focus on solving the macroscopic problem. 
We present a simulation for the macroscopic model and for the case that our domain is 
a cardioid
in order to capture the behavior of the clogging process close to the singular point.
We use a finite element method to solve the two-dimensional version of the field equations 
 \eqref{overall-u1_in} together with \eqref{mod1hb} and \eqref{alter-v} accompanied by appropriate boundary and initial conditions.

Note that here we consider Robin boundary conditions everywhere in the boundary of the domain,
non-homogeneous for species $u_1$ and homogeneous for the rest of the species. Specifically, we take in equation $\eqref{mod1b}$, for $u_i^b$, 
\bse
 \be \label{mod1b_L}
  u_i^b(x_1,x_2)=
  \left\{\begin{array}{cc}
   1 & \mbox{for}\, i=1,\\
  0         &  \mbox{for}\, i\neq 1,
 \end{array}\right., \quad (x_1, x_2)\in \partial \Omega,
 \ee
  \be
  u_i(0)=u_{i0}=0. \quad \text{in}\ \ \Omega.\label{overall-u3_L}
  \ee
\ese
In this way, we simulate a constant inflow of $u_1$ through the boundary of $\partial \Omega$.

We solve the macroscopic problem by implementing a finite element scheme in the MATLAB finite element  package ''\texttt{PDE Toolbox}"
The entries of the diffusion tensor, as previously discussed, are given by the solution of the cell problems with $y$-periodic boundary conditions in the case of a growing symmetric ellipse (see the third graph in Figure \ref{FigCellper}). Thus, the diffusion tensor can be used via interpolation in the macroscopic problem.

Moreover, we consider three species. Our model needs a quite large number of parameters and we take them as follows: $(d_1,\,d_2,\,d_3)=(1,\, 0.5,\, 0.9)$, $(a_1,\,a_2,\,a_3)=(.9,.9,.9)$, $(b_1,\,b_2,\,b_3)=(1,1,1)$,
$\gamma_{ij}=10$, $i,j=1,\ldots 3$, $b_r=1$.
 The shape of the macroscopic domain is given by the equation $( 2(1+\cos(s))\,\cos(s),\,2 (1+\cos(s))\sin(s))$, $s\in [0,\,2\pi]$.

 The result is presented initially in Figure \ref{FigMacroERea_Cardshape} where the
distribution of the long axis of the growing ellipses $R_a$ is plotted for four time frames. The simulation time interval is $[0,T]$ for $T=3$ and the solution is 
presented in $t=0.1,\,    0.85,\,    1.6,\,    2.35$.
\begin{figure}[htb]
\vspace*{-1cm}
\begin{center}
\includegraphics[bb= 330 230 250 600, scale=.7]{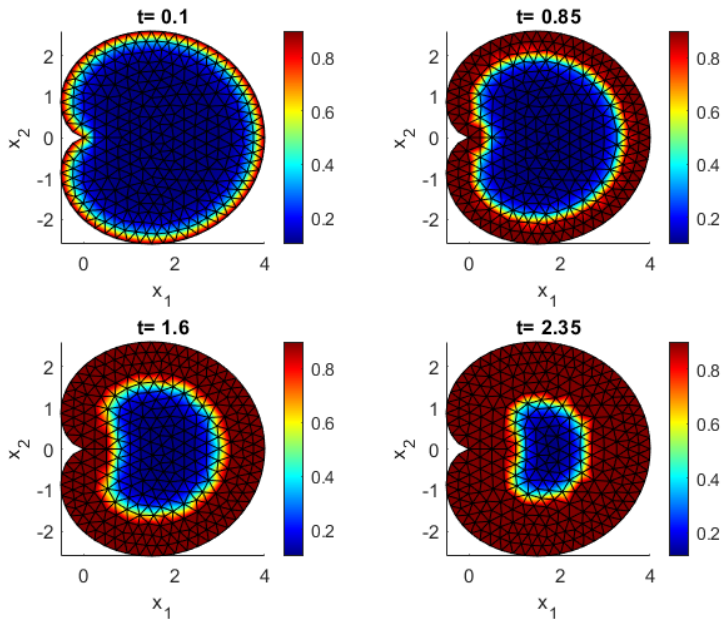}
 \vspace*{-1cm}
\end{center}
\caption{\it Time frames presenting the evolution of the long axis of the ellipse in the cell.}\label{FigMacroERea_Cardshape}
\end{figure}

Clogging is apparent as the distribution of $R_a$ increases and reaches its maximum value throughout the domain. At the end of the simulation, all the domain becomes clogged. 
In this case, the evolution of clogging seems to smooth out the singularity of the macroscopic domain. We notice that a similar behaviour occurs for other choices of macroscopic domains; see, for instance, the domain presented in the next paragraph.

\subsection{Simulations for an L-shaped macroscopic domain} 
A similar simulation is also presented for an $L-$shaped domain.
More specifically we consider an $L-$shaped set of the form
$\Omega =\left([-1,1]\times [-1,0]\right) \cup \left([0,1]\times [0,1]\right)$.

Using the same setting as before regarding the diffusion tensor and the same parameter values 
we obtain the result presented in Figure \ref{FigMacroEReaLsh}
Note that the same boundary and initial conditions,
\eqref{mod1b_L}  and \eqref{overall-u3_L} as in the  cardioid 
domain  simulation is used.
The long axis of the growing ellipses $R_a$ is plotted in four time frames. The simulation time interval is $[0,T]$ for 
$T=1.2$ and the solution is presented in $t=0.1,\,    0.4,\,    0.7,\,    1$.


 Again, as for the cardioid case, we can observe the clogging behaviour as the distribution of $R_a$ increases and reaches its maximum value throughout the domain while at the end of the simulation, all the domain is clogged. 
\begin{figure}[htb]
\begin{center}
\includegraphics[bb= 330 230 250 600, scale=.7]{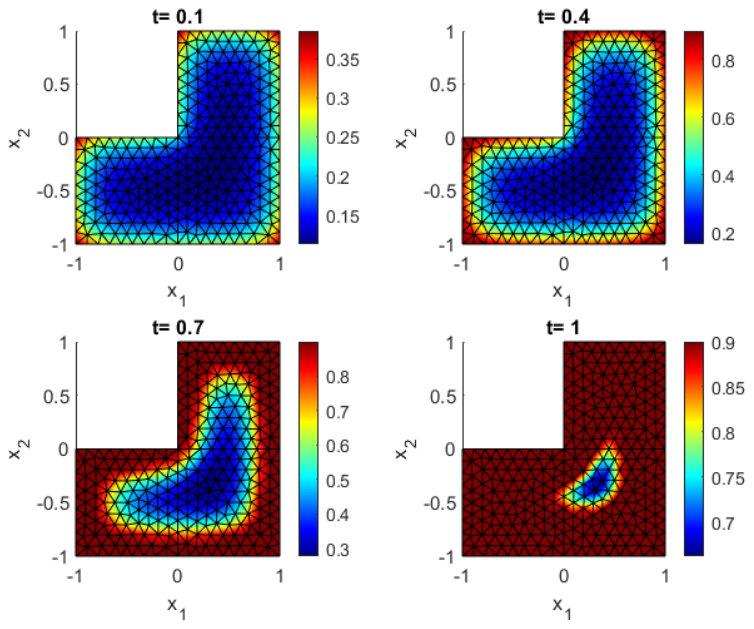}
 \vspace*{-1cm}
\end{center}
\caption{\it Time frames presenting the evolution of the long axis of the ellipse in the cell.}\label{FigMacroEReaLsh}
\end{figure}
We observe that the cells near the top bottom and left corners are filled faster and therefore move towards clogging than the inner corner at $(0,0)$. This can be seen especially in the first two graphs at $t=0.1$ and 
 $t=0.4$ but also in the third graph for $t=0.7$. The deposition process slows down in the presence of the concave corner.
\subsection{Simulations for an L-shaped macroscopic domain with non-uniform initial distribution} 
In the next simulation, we revisit  the  same $L-$shaped macroscopic domain but with  non uniform initial distribution of the elliptical core sizes in the microscopic domain.
Specifically, the distribution of the long axis of the growing ellipses $R_a$ at time $t=0$ is taken to have the form
$R_a=\frac12 \max (R0,\sin(\omega x_1 x_2))$ for $R_0=0.01,\,\omega=10$. In this way, we create a kind of "barrier" with fewer open pores  in the three square subdomains of the $L$-shaped domain,  as it is visible in the upper left frame in Figure \ref{FigMacroEReaLshNU}. We want to capture the effect of this non-uniformity in the appearance and evolution of the clogging. The rest of the setting and the parameters used are exactly the same as in the previous simulations.
\begin{figure}[htb]
\vspace*{-6cm}
\begin{center}
\includegraphics[bb= 330 230 250 600, scale=.7]{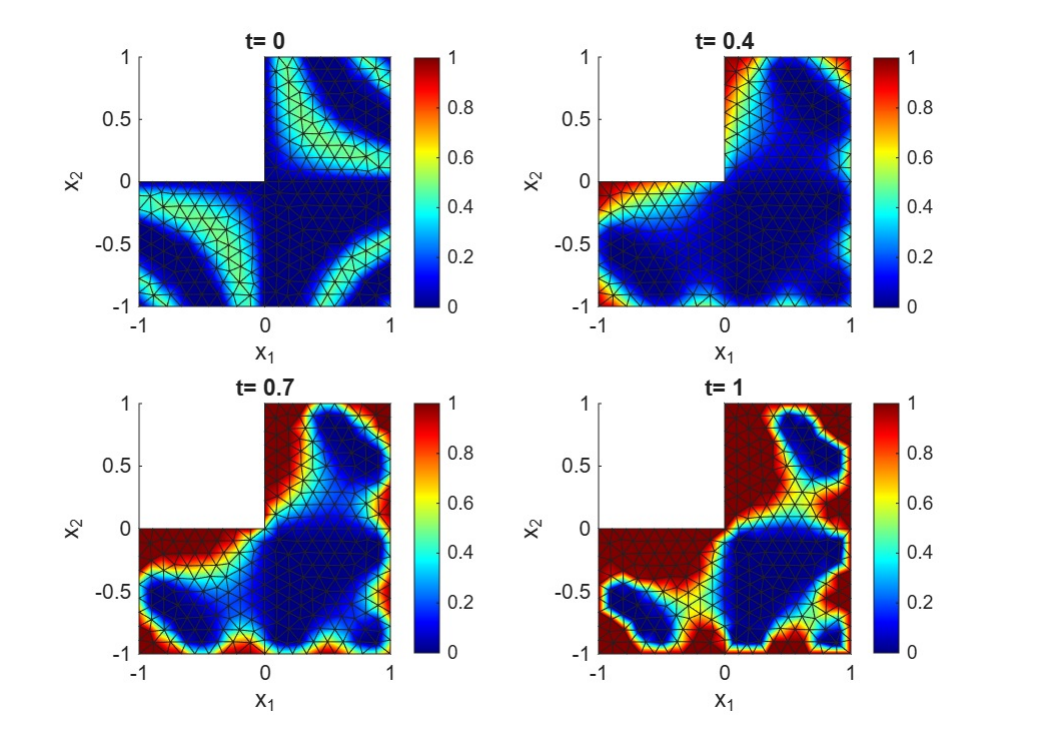}
 \vspace*{5cm}
\end{center}
\caption{\it Time frames presenting the evolution of the long axis of the ellipse in the cell in the case of an initial non-uniform distribution.}\label{FigMacroEReaLshNU}
\end{figure}
We observe that the inflow through the boundary creates clogging areas in front of the aforementioned "barriers" since we have increased accumulation of colloids  in these areas.
Additionally, the convex corners are   susceptible to clogging while the concave one is much less inclined to clog.


\section{Conclusion}\label{conclusion}
We presented a two-scale reaction-diffusion system modeling the deposition, fragmentation and agglomeration of populations colloidal particles in porous media. The shape of the microstructure of the media is allowed to change in both space and time coordinates. We establish conditions under which weak solutions to our model exist and are unique. Finally, we compute them numerically for selected model parameters and geometries. 

We point out that, benefitting of the structure of our model,  we can detect clogging effects by means of numerical simulations.
Specifically,  the significant entries of the diffusion tensor can be calculated through the numerical solution of the cell problems. This information can be used to simulate  the macroscopic process and to infer that the evolution of clogging seems to smooth out  singularities of the macroscopic domain and, moreover, that convex corners are more susceptible to clogging than the concave ones. Moreover we observe that inflow through the boundary creates clogging areas in front of regions with smaller porosity.

The numerical detection of transport defects can become of practical use for a large variety of applications  ranging from controlling drug-delivery tasks and improving membrane filtration of pollutant agents,  to guiding self-healing of concrete-based materials especially when clogging effects can be captured numerically in 3D. 
If such computations become feasible and efficient, it would be valuable to use 3D numerical simulations across a wide range of initial conditions and microstructural distributions. This would enable, for instance, a systematic investigation of carbonation-induced self-healing in Portland concrete modified with slag and fly ash—materials known to enhance durability, increase strength, reduce heat of hydration, and lower permeability. We plan to explore this scenario further in a follow-up project.

\section*{Acknowledgments}
AM thanks the Swedish Research Council (project nr. 2024-05606)  for financial support. 




\end{document}